\documentclass{amsart}
\theoremstyle{plain}
\newtheorem{thm}{Theorem}

\newtheorem{lem}[thm]{Lemma}
\newtheorem{cor}[thm]{Corollary}
\theoremstyle{remark}
 \newtheorem{rem}[thm]{Remark}

\errorcontextlines=0 \numberwithin{equation}{section}

\begin{document}

\title[Curvature tensor under the Ricci-Harmonic flow]
{Curvature tensor under the Ricci-Harmonic flow}

\author{Anqiang Zhu, Liang Cheng$^1$}

\address{Anqiang Zhu, Department of Mathematical Sciences, WuHan University,
WuHan 430072, P. R. China} \email{anqiangzhu@yahoo.com.cn}

\address{Liang Cheng, Department of mathematics, Huazhong Normal University, WuHan 430079, P. R. China}
\email{chengliang@gmail.com}

\dedicatory{}
\date{}


\begin{abstract}
We prove that if the Ricci curvature is uniformly bounded under the
Ricci-Harmonic flow for all times $t\in [0,T)$, then the curvature
tensor has to be uniformly bounded as well.
\end{abstract}

\thanks{$^1$= corresponding author; The research is partially supported by the National Natural Science
Foundation of China 10631020 and SRFDP 20090002110019}

\maketitle
\section{introduction}
The Ricci flow has been a useful tool in the study of geometry
problem. It played an important part in the proof of the Poincare
Conjecture and Geometric Conjecture. Since then, geometric flows
attract more attention. In \cite{B.List}, B.List studied an extended
Ricci flow
$$\label{equation 1}
\left\{
    \begin{array}{ll}
      \frac{\partial }{\partial t}g_{ij}=-2R_{ij}+\alpha\phi_{i}\phi_{j}, & \hbox{$t\in [0,T)$;} \\
      \frac{\partial }{\partial t}\phi=\Delta \phi, & \hbox{$t\in [0,T)$.}
    \end{array}
  \right.
$$
In his dissertation, B.List said this flow has some applications in
general relativity. The nonlocal collapse and monotonicity of some
entropy have been proved. Some gradient estimate and local existence
of the flow are obtained. As long as the Riemannian curvature tensor
is bounded, the flow will exists and the gradient and Hession of
$\phi$ are also bounded, \cite{B.List}. R.Muller also introduced a
new geometric flow, called Ricci-Harmonic flow,
$$\label{equation 1}
\left\{
    \begin{array}{ll}
      \frac{\partial }{\partial t}g_{ij}=-2R_{ij}+\alpha\phi_{i}\phi_{j}, & \hbox{$t\in [0,T)$;} \\
      \frac{\partial }{\partial t}\phi=\tau_{g(t)} \phi, & \hbox{$t\in [0,T)$.}
    \end{array}
  \right.
$$
where $\phi:(M,g(t))\rightarrow (N,h)$ is a map between two
Riemannian manifold, $\tau_{g}\phi=trace \nabla d\phi$. The extended
Ricci flow of B.List is a special case of Ricci-Harmonic flow.
Similar to the extended Ricci flow, this flow will also exist if the
Riemannian curvature tensor is bounded.

On the other hand, there are many results about controlling the
Riemannian curvature tensor under Ricci flow. In \cite{natasa}, by a
blow up argument, N.Susem showed that Ricci curvature uniformly
bounded on $M\times [0,T)$, where $T<\infty$ is enough to control
the norm of Riemannian curvature tensor on closed manifold. Then
L.Ma and L.Cheng improveed the result to the Ricci flow on complete
noncompact manifold. R.Ye \cite{ye} and B.Wang \cite{wang}, by
different arguments, proved that the norm of the Riemannian
curvature tensor can be controlled by the bound of
$\|Rm\|_{\frac{n+2}{2}}=(\int_{0}^{T}\int_{M}|Rm|^{\frac{n+2}{2}}d\mu
dt)^{\frac{2}{n+2}}$. In this paper, we will prove that if the Ricci
curvature is uniformly bounded, the flow can continue.

The following are the main results of this paper.

\begin{thm}\label{Muller}
Let $(M,g(t),\phi(t)),t\in [0,T)$ be a solution to the Ricci
Harmonic flow on a Riemannian manifold. Suppose $T<\infty$  and
Ricci curvature is uniformly bounded under the flow. Then the
Riemannian curvature tensor $|Rm|$ stays unoformly bounded under the
flow.
\end{thm}

In the following, we will prove \ref{Muller}, Since the Extended
Ricci flow is a special case of the Ricci Harmonic flow. Note that
the proof doesn't depend on the compact of the manifold. We just
take a closed manifold for simple.

If the manifold is compact, we also have the following result.
\begin{thm}
Let $(M,g(t),\phi(t)),t\in [0,T)$, where $T<\infty$, be a solution
to the Ricci-Harmonic flow on a closed manifold satisfying
$$
\|R\|_{\frac{n+2}{2}}=(\int_{0}^{T}\int_{M}|R|^{\frac{n+2}{2}}d\mu
dt)<\infty,
$$
and
$$
\|W\|_{\frac{n+2}{2}}=(\int_{0}^{T}\int_{M}|W|^{\frac{n+2}{2}}d\mu
dt)<\infty,
$$
\end{thm}
Then $\sup_{M\times [0,T)}\|Rm\|<\infty$.

In the following, we denote $S_{ij}=R_{ij}-\alpha \phi_{i}\phi_{j}$,
$S=R-\alpha|\nabla \phi|^{2}$.

\section{Preliminary}
\begin{lem}\label{Gradient estimate}
Under the same condition of the theorem \ref{Muller}, we have
$$
|\nabla \phi|(t)<C, t\in [0,T).
$$
where $C$ is independent of $t$.
\end{lem}
\begin{proof}
Since $Ric $ curvature is uniformly bounded along the flow at $t\in
[0,T)$, the Scale curvature is also uniformly bounded. From
\cite{Muller}, we have
\begin{eqnarray*}
(\partial_{t}-\Delta)S=2|S_{ij}|^{2}+2\alpha |\tau_{g}\phi(t)|^{2}.
\end{eqnarray*}
So the minimum of $S(t)$ is nondecreasing along the flow.
$$
S(t)=R-\alpha |\nabla \phi|^{2}\geq C.
$$
Since $R$ is uniformly bounded, $|\nabla \phi(t)|^{2}\leq C$, where
$C$ is independent of $t$.

\end{proof}

Since $Ric$ curvature and $|\nabla \phi(t)|^{2}$ is uniformly
bouhnded, we will control the variation of the distance along the
Extended Ricci flow.
\begin{lem}\cite{B.List},\cite{W}
Under the same condition of theorem \ref{Muller}. For all
$\delta>0$, there exists an $\eta>0$, such that if
$|t-t_{0}|<\eta,t,t_{0}\in [0,T)$, then
$$
|d_{g(t)}(q,q^{'})-d_{g(t_{0})}(q,q^{'})|\leq \delta
d_{g(t_{0})}(q,q^{'})
$$
for all $q,q^{'}\in M$.
\end{lem}

\begin{proof}
\begin{eqnarray*}
|\log g(t)(V,V)-\log
g(t_{0})(V,V)|&=&|\int_{t_{0}}^{t}\frac{\partial}{\partial
t}\log g(s)(V,V)ds|\\
&=&|\int_{t_{0}}^{t}\frac{(-2Ric(V,V)+2\alpha <\nabla
\phi,V>^{2})}{|V|_{g(s)}^{2}}ds|\\
&\leq &\int_{t_{0}}^{t}\frac{2C|V|^{2}_{g_{s}}+2\alpha|\nabla
\phi|_{g(s)}^{2}|V|_{g(s)}^{2}}{|V|^{2}_{g_{s}}}ds\\
&\leq &C(t-t_{0})
\end{eqnarray*}

So we have
$$
e^{-C|t-t_{0}|}|V|^{2}_{g(t_{0})}\leq |V|^{2}_{g(t)}\leq
e^{C|t-t_{0}|}|V|^{2}_{g(t_{0})}
$$

Suppose $\gamma_{1}$ is a minimal geodesic from $p$ to $p^{'}$ with
respect to metric $g(t_{0})$.
$$
d_{g(t)}(p,p^{'})\leq \int |\dot{\gamma_{1}}|_{g(t)}(s)ds\leq \int
e^{C|t-t_{0}|/2}|\dot{\gamma_{1}}|_{g(t_{0})}(s)ds=e^{C|t-t_{0}|/2}d_{g(t_{0})}(p,p^{'}).
$$
Similarly, Let $\gamma_{2}$ be a minimal geodesic from $p$ to
$p^{'}$ with respect to metric $g(t)$. Then
$$
d_{g(t)}(p,p^{'})= \int |\dot{\gamma_{2}}|_{g(t)}(s)ds\geq \int
e^{-C|t-t_{0}|/2}|\dot{\gamma_{2}}|_{g(t_{0})}(s)ds\geq
e^{-C|t-t_{0}|/2}d_{g(t_{0})}(p,p^{'}).
$$

$$
|d_{g(t)}(p,p^{'})-d_{g(t_{0})}(p,p^{'})|\leq
(e^{C|t-t_{0}|/2}-1)d_{g_{t_{0}}}(p,p^{'}).
$$
The lemma will be satisfied if we take $\eta=\frac{2\ln
(\delta+1)}{C}$.
\end{proof}

\begin{cor}\label{distance estimate}
Under the condition of theorem \ref{Muller}. For all $\rho>0$,
$$
B_{g(t)}(0,r(t)\rho)\subset B_{g(0)}(0,\rho),
$$
$$
B_{g(0)}(0,r(t)\rho)\subset B_{g(t)}(0,\rho),
$$
where $r(t)=e^{-Ct/2}$.
\end{cor}

\begin{proof}
Suppose $x\in B_{g(t)}(o,r(t)\rho)$. Then $r(t)\rho\geq
d_{g(t)}(o,x)\geq e^{-Ct/2}d_{g(0)}(o,x)$. So $d_{g(0)}(o,x)\leq
\rho$. Similarly for $B_{g(0)}(0,r(t)\rho)\subset B_{g(t)}(0,\rho)$.
\end{proof}

\begin{lem}\label{volume estimate}
Under the same condition of theorem \ref{Muller}. For all
$\epsilon>0$, there exists an $\delta (\epsilon,C)>0$, such that if
$t\in [t_{0},t_{0}+\delta]$ then
$$
Vol_{g(t)} B_{g(t)}(x,r)\geq
(1-\epsilon)^{\frac{n}{2}}Vol_{g(t_{0})}B_{g(t_{0})}(x,\frac{r}{1+\epsilon}).
$$
\end{lem}

\begin{proof}
By the corollary \ref{distance estimate}, for all $\epsilon>0$,
there exists $\delta_{1}>0$, such that if $t\in
[t_{0},t_{0}+\delta_{1}]$, we have
$$
B_{g(t_{0})}(x,\frac{r}{1+\epsilon})\subset B_{g(t)}(x,r).
$$

Since the Ricci curvature is uniformly bounded in $[0,T)$, by lemma
\ref{Gradient estimate}, we have $|S|<C$ is uniformly bounded.
\begin{eqnarray*}
\frac{d}{dt}\int_{B_{g(t_{0})}(x,\frac{r}{1+\epsilon}
)}dv=\int_{B_{g(t_{0})}(x,\frac{r}{1+\epsilon})}-S dv
\end{eqnarray*}
i.e.
$$
|\frac{d}{dt}Vol_{g(t)}B_{g(t_{0})}(x,\frac{r}{1+\epsilon})|<C
Vol_{g(t)}B_{g(t_{0})}(x,\frac{r}{1+\epsilon}).
$$
\end{proof}
We have
$$
Vol_{g(t)}B_{g(t_{0})}(x,\frac{r}{1+\epsilon})\geq
e^{-C(t-t_{0})}Vol_{g(t_{0})}B_{g(t_{0})}(x,\frac{r}{1+\epsilon})
$$
There exists a $\delta_{2}>0$, such that $e^{-C(t-t_{0})}\geq
(1-\epsilon)^{\frac{n}{2}}$, if $t\in [t_{0},t_{0}+\delta_{2}]$.

Take $\delta=\min\{\delta_{1},\delta_{2}\}$, for $t\in
[t_{0},t_{0}+\delta]$
$$
Vol_{g(t)} B_{g(t)}(x,r)\geq
(1-\epsilon)^{\frac{n}{2}}Vol_{g(t_{0})}B_{g(t_{0})}(x,\frac{r}{1+\epsilon}).
$$

\begin{lem}\label{compacness}
Let $\{(M_{i}^{n},g_{i}(t),\phi_{i}(t),x_{i})\}_{i=1}^{\infty},t\in
[0,T]$ be a sequence of the Ricci-Harmonic flow on complete
manifolds such that $\sup_{M_{i}\times
[0,T]}|Rm(g_{i}(t))|_{g_{i}(t)}\leq 1$, $\sup_{x\in
M}|\phi_{i}(x,0)|<C$ where $C$ is independent of $i$. Let
$\psi_{i}=\exp_{x_{i},g_{i}(0)}$ be the exponential map with respect
to metric $g_{i}(0)$ and $B(o_{i},\frac{\pi}{2})\subset
T_{x_{i}}M_{i}$ equipped with metric $\tilde{g_{i}}(t)\doteq
\psi_{i}^{*}g_{i}(t), \tilde{\phi_{i}}=\phi_{i}\circ \psi_{i}$. Then
$(B(o_{i},\frac{\pi}{2}), \tilde{g_{i}}(t),\tilde{\phi_{i}}(t),
o_{i})$ subconverges to a Ricci-Harmonic flow
$(B(o,\frac{\pi}{2}),\tilde{g}(t),\tilde{\phi},o)$ in $C^{\infty}$
sense, where $B(o,\frac{\pi}{2})\subset R^{n}$ equipped with metric
$\tilde{g}(t)$.
\end{lem}

\begin{proof}
Since $\tilde{g_{i}}(t)\doteq\psi^{*}_{i}g_{i}(t)$, we have
$\sup_{B(o_{i},\frac{\pi}{2})\times
[0,T]}|Rm(\tilde{g_{i}}(t))|_{\tilde{g_{i}}(t)}\leq 1$. On
$B(o_{i},\frac{\pi}{2})$, we have $inj(o_{i},\tilde{g_{i}}(0))\geq
\frac{\pi}{2}$ \cite{LiMa}. Then the lemma follows from the
compactness theorem in \cite{W},\cite{B.List}.
\end{proof}

\section{Blow up analysis}

\begin{proof}
By the Maximum principle, we have
$$
\inf_{x\in M}\phi(x,0)\leq \phi(x,t)\leq \sup_{x\in M}\phi(x,0)
$$
for $t\in [0,T)$.
 In the following, we consider the pull back
metrics on tangent spaces. Suppose the Ricci-Harmonic flow blows up
at a finite time $T$. That means there exist sequences
$t_{i}\rightarrow T$ and $p_{i}\in M$ such that
$$
Q_{i}=|Rm|(p_{i},t_{i})\geq C^{-1}\max_{M\times [0,t_{i}]}|Rm|(x,t).
$$
where $C>1$ and $Q_{i}\rightarrow \infty$ as $i\rightarrow \infty$.
By lemma \ref{compacness}, the scaled flow
$(B(o_{i},\frac{\pi}{2C}), g_{i}(t), \phi_{i}(t))$ will converge to
a complete ancient solution
$(B(o,\frac{\pi}{2C}),\overline{g}(t),\overline{\phi})$ of the
Ricci-Harmonic flow, where $g_{i}(t)=Q_{i}g(t_{i}+\frac{t}{Q_{i}}),
\phi_{i}(t)=\phi(t_{i}+\frac{t}{Q_{i}})$. Since $S(t)=R(t)-\alpha
|\nabla\phi(t)|^{2}$ is uniformly bounded,
$S_{i}(t)=\frac{1}{Q_{i}}(R(g(t_{i}+\frac{t}{Q_{i}}))-\alpha |\nabla
\phi|^{2}(g(t_{i}+\frac{t}{Q_{i}})))$. As $i\rightarrow \infty$,
$S_{i}(t)\rightarrow 0$. Our ancient solution
$(B(o,\frac{\pi}{2C}),\overline{g}(t),\overline{\phi})$ has $S=0$.
The evolution of $S$ is
$$
(\partial_{t}-\Delta)S=2|S_{ij}|^{2}+2\alpha|\Delta\phi|^{2}.
$$
This imply that $|S_{ij}|=0,|\Delta\phi|=0$. i.e. $Ric=\alpha\nabla
\phi\bigotimes\nabla \phi\geq 0$.

In the following, we will analyze the volume of the limit
$(B(o,\frac{\pi}{2C}),\overline{g}(t),\overline{\phi})$.

Take $r>0$, since the convergence is $C^{\infty}$, the volume is
convergence,
$$
\frac{Vol B(p,r)}{r^{n}}=\lim_{i\rightarrow \infty}\frac{Vol_{i}
B_{i}(p_{i},r)}{r^{n}}
$$
where the volume and the ball $B(p,r)$ on the LHS are considered in
metric $\overline{g}$, while the RHS are considered in metric
$g_{i}(0)=Q_{i}g(t_{i})$. So

$$
\frac{Vol B(p,r)}{r^{n}}=\lim_{i\rightarrow
\infty}\frac{Vol_{g(t_{i})} B_{g(t_{i})}(p_{i},r
Q_{i}^{-1/2})}{(Q_{i}^{-1/2}r)^{n}}
$$

For $\forall \epsilon>0$, there exists $N>0$, such that if $i>N$, we
have
$$
\frac{Vol B(p,r)}{r^{n}}\geq \frac{Vol_{g(t_{i})}
B_{g(t_{i})}(p_{i},r
Q_{i}^{-1/2})}{(Q_{i}^{-1/2}r)^{n}}-\frac{\epsilon}{2}.
$$

For such $\epsilon>0$, there exists a $\delta>0$ from lemma
\ref{volume estimate}. We take $t_{0}>T-\delta$. For $i$
sufficiently large, $0<t_{i}-t_{0}<\delta$,

$$
\frac{Vol B(p,r)}{r^{n}}\geq  (1-\epsilon)^{n/2}\frac{Vol_{g(t_{0})}
B_{g(t_{0})}(p_{i},\frac{1}{1+\epsilon}r
Q_{i}^{-1/2})}{(Q_{i}^{-1/2}r)^{n}}-\frac{\epsilon}{2}.
$$

Note that volume has the expansion within the injective radius of
$p$(see \cite{Gallot} theorem 3.98),
$$
Vol B(p,r)=\omega_{n}r^{n}(1-\frac{R(p)}{6(n+2)}r^{2}+o(r^{2})).
$$

Let $i\rightarrow \infty$, we have
$$
\frac{Vol B(p,r)}{r^{n}}\geq
\omega_{n}\frac{(1-\epsilon)^{n/2}}{(1+\epsilon)^{n}}-\frac{\epsilon}{2}.
$$
Let $\epsilon\rightarrow 0$, we have
$$
\frac{Vol B(p,r)}{r^{n}}\geq \omega_{n}.
$$
Since $Ric\geq 0$, by the Bishop-Gromov volume comparison, we have
$(N,g)$ is a flat manifold, which is contradict to the
$|Rm|(p,0)=1$.

\end{proof}

\begin{proof}
From Theorem 8.6 in \cite{B.List}, since $M$ is compact, the limit
solution is in fact a ancient solution to Ricci flow. The proof is
the similar to the proof in \cite{M-C}.
\end{proof}

\begin{rem}
The blow up can also be applied to the complete noncompact
Riemannian manifold as in \cite{M-C}.
\end{rem}

\end{document}